\documentclass{amsart}

\usepackage[colorlinks=true, urlcolor=blue,bookmarks=true,bookmarksopen=true,citecolor=blue,hypertex]{hyperref}
\usepackage{graphicx}
\usepackage{enumerate}

\usepackage{amscd}
\usepackage{amssymb}
\usepackage{amsxtra}
\usepackage{amsmath}
\usepackage{enumerate}
\usepackage{mathrsfs}

\usepackage{color}

\usepackage{amsmath,amsfonts,amssymb}
\usepackage{verbatim}
\usepackage[hmargin = 4cm,vmargin = 2.5cm]{geometry}
\usepackage{epsfig}
\usepackage{amsthm}
\usepackage{bbold}

\theoremstyle{plain}

\newtheorem{thm}{Theorem}

\newtheorem{lem}[thm]{Lemma}

\theoremstyle{definition}

\theoremstyle{remark}

\newtheorem{rmk}[thm]{Remark}

\theoremstyle{plain}

\newcommand{\Id}{{{\mathchoice {\rm 1\mskip-4mu l} {\rm 1\mskip-4mu l}
      {\rm 1\mskip-4.5mu l} {\rm 1\mskip-5mu l}}}}
\newcommand\supp{\operatorname{supp}}

\def\NN{\mathbb{N}}
\def\ZZ{\mathbb{Z}}

\def\RR{\mathbb{R}}

\def\Id{\mathbb{1}}

\def\dd{\mathrm{d}}
\def\Ham{Ham}
\def\Cal{Cal}

\def\len{\operatorname{len}}
\def\Area{\operatorname{Area}}

\def\dist{\operatorname{dist}}
\def\diam{\operatorname{diam}}

\begin{document}

\pagestyle{headings}

\bibliographystyle{alphanum}

\title{Quasimorphisms on surfaces and continuity in the Hofer norm}

\thanks{The author was supported by the Azrieli Fellowship.}

\author{Michael Khanevsky}
\address{Michael Khanevsky, Mathematics Department,
Technion - Israel Institute of Technology
Haifa, 32000,
Israel}
\email{khanev@math.technion.ac.il}

\begin{abstract}
There is a number of known constructions of quasimorphisms on Hamiltonian groups. We show that on surfaces many of these quasimorphisms are not compatible 
with the Hofer norm in a sense they are not continuous and not Lipschitz. The only exception known to the author is the Calabi quasimorphism on a sphere ~\cite{En-Po:calqm} and the induced
quasimorphisms on genus-zero surfaces (e.g. ~\cite{Bi-En-Po:Calabi-qm}).
\end{abstract}

\maketitle

\section{Introduction and results}

Let $(M, \omega)$ be a symplectic manifold. The Hamiltonian group of $M$ admits Hofer's metric which, roughly speaking, measures mechanical energy needed to deform one
Hamiltonian into another. This metric has many nice properties - it is bi-invariant and on closed manifolds the induced geometry is a natural one: any other bi-invariant 
$C^\infty$-continuous Finsler metric is equivalent to Hofer's (~\cite{B-O:equiv}). However until now the large scale geometry not well understood as there is a very limited set of tools available to 
estimate distances. It is particularly difficult to produce lower bounds.

Ideally, one would like to construct an invariant that estimates the distance and at the same time respects the group structure - namely, a homomorphism. In the case  
when the symplectic form $\omega$ is exact the Calabi homomorphism is such a tool as it is 1-Lipschitz with respect to the metric. However, by a well known result 
~\cite{Ba:simple-diff}, $\Ham(M)$ is either simple (when $M$ is closed) or it contains a
simple subgroup of codimension 1, therefore all homomorphisms factor through Calabi. As an attempt to relax the constraints one may consider quasimorphisms: maps $\Ham(M) \to \RR$ which are additive up to a bounded
defect. When one is interested in coarse estimates of geometry, the defect does not play a significant role. There is a number of known constructions of quasimorphisms 
for various manifolds and the main question is whether they can be utilized to extract geometric information. For example, the Calabi quasimorphism on $\Ham(S^2)$ and certain other manifolds (~\cite{En-Po:calqm}) 
arises as a Floer-theoretic spectral invariant and is naturally Lipschitz with respect to the Hofer metric. There is a series of other constructions that consider topological 
invariants of orbits of a Hamiltonian flow. Some of these quasimorphisms can be seen as generalizations of the Calabi homomorphism so there was a hope they may inherit 
metric properties. Unfortunately, this is not the case and all constructions known to the author except for the Calabi quasimorphisms and the induced ones are not Lipschitz with respect to 
Hofer's metric.
In this article we review two families of quasimorphisms on surfaces constructed by Polterovich and Gambaudo-Ghys and show that they are neither Lipschitz nor continuous in the Hofer metric. 
Some results can be generalized to symplectic manifolds of higher dimension. 

The situation is different when one considers some other (not bi-invariant) metrics on $\Ham (M)$. For example, the quasimorphisms
mentioned above are continuous in the $L^p$ norm (see ~\cite{Br-Sh:Lp-diam-S2, Br-Sh:Lp-geom-S2, Br-Sh:Lp-geom-surf}).

\medskip

\emph{Acknowledgements:}
The author is grateful to L. Buhovski, L. Polterovich and E. Shelukhin for useful discussions and comments.

\section{Preliminaries}\label{S:def}

Let $(M, \omega)$ be a symplectic manifold, $g$ a Hamiltonian diffeomorphism with compact support in $M$.
The Hofer norm $\|g\|$ (see ~\cite{Hof:TopProp}) is defined by 
\[
  \|g\| = \inf_G \int_0^1 \max \left(G(\cdot, t) - \min G(\cdot, t) \right) \dd t 
\]
where the infimum goes over all compactly supported Hamiltonian functions $G : M \times [0,1] \to \RR$
such that $g$ is the time-1 map of the induced flow. The Hofer metric is given by
\[
  d_H (g_1, g_2) = \| g_1 g_2^{-1}\|.
\]

\medskip

Let $G$ be a group. A function $r : G \to \RR$ is called a \emph{quasimorphism} if there exists 
a constant $D$ (called the \emph{defect} of $r$) such that $|r(fg) - r(f) - r(g)| < D$ for all $f, g \in G$. The quasimorphism $r$ 
is called \emph{homogeneous} if it satisfies $r(g^m) = mr(g)$ for all $g \in G$ and $m \in \ZZ$.
Any homogeneous quasimorphism satisfies $r(fg) = r(f) + r(g)$ for commuting elements $f, g$. Every quasimorphism
is equivalent (up to a bounded deformation) to a unique homogeneous one ~\cite{Ca:scl}. A quasimorphism is called \emph{genuine} if it is not a homomorphism.

\medskip

\begin{lem}
  Let $G$ be a Lie group equipped with a bi-invariant path metric, $r: G \to \RR$ a homogeneous quasimorphism. Then $r$ is Lipschitz 
  if and only if it is continuous at the identity.
\end{lem}
\begin{proof}
  The fact that Lipschitz property implies continuity is obvious. We show the opposite direction.
  Suppose $r$ is continuous at the identity, namely, given $\varepsilon > 0$, there exists $\delta = \delta(\varepsilon) > 0$ such that any $g \in G$ with 
  $\dist(g, \Id) = \|g\| \leq \delta$ satisfies $|r(g)| < \varepsilon$. Fix an $\varepsilon > 0$ and select an appropriate $\delta = \delta(\varepsilon)$.
  
  \medskip
  
  The Lipschitz property holds on a large scale:  
  let $g, h \in G$ with $g$ arbitrary and $\|h\| \geq \delta$. By the triangle inequality,
  \[
	|r(hg) - r(g)| \leq |r(h)| + D
  \]
  where $D$ denotes the defect of $r$.
  
  Pick a path $h_t$ connecting $h$ to the identity with $\len(h_t) \leq \|h\|+\delta$. 
  The path $h_t$ can be cut into $N := \left\lceil \frac{\len(h_t)}{\delta} \right\rceil < \frac{\|h\|}{\delta}+2$ arcs of length at most $\delta$, 
  hence $h$ can be presented as a composition of $N$ elements in the $\delta$-neighborhood of the identity. Therefore 
  \[
	|r(h)| = |r(h_1 \cdot \ldots \cdot h_N)| < N \cdot \left(\varepsilon + D\right) < \left(\frac{\|h\|}{\delta} + 2\right)\left(\varepsilon + D\right) \leq 
	\|h\| \cdot \frac{3}{\delta} (\varepsilon + D).
  \]
  Finally,
  \[
	|r(hg) - r(g)| \leq |r(h)| + D <  \|h\| \cdot \frac{3}{\delta} (\varepsilon + D) + \|h\| \cdot \frac{D}{\delta}
  \]
  and the property holds since $\|h\| = \dist (h, \Id) = \dist (hg, g)$ by bi-invariance of the metric.
  
  \medskip
  
  We show that the Lipschitz property holds also locally: pick $g, h \in G$ with $g$ arbitrary and $h$ in the $\delta$-neighborhood of the identity.
  Denote
  \[
	  k := \left\lfloor \frac{\delta}{\|h\|} \right\rfloor > \frac{\delta}{2 \|h\|}.
  \]
  
  Then by homogenuity and the triangle inequality 
  \begin{eqnarray*}
	|r(hg) - r(g)| &=& \frac{1}{k} \cdot \left| r\left((hg)^k\right) - r \left(g^k\right) \right| = \\
					& &	 \frac{1}{k} \cdot \left| r\left(h \cdot ghg^{-1} \cdot g^2hg^{-2} \cdot \ldots \cdot g^{k-1}hg^{1-k} \cdot g^k\right) - r \left(g^k\right) \right| \leq \\
					& &	 \frac{1}{k} \cdot \left( \left| r\left(h \cdot ghg^{-1} \cdot g^2hg^{-2} \cdot \ldots \cdot g^{k-1}hg^{1-k} \right)\right| + D \right)
  \end{eqnarray*}
  By bi-invariance $\|g^n h g^{-n}\| = \|h\| \leq \frac{\delta}{k}$ and the triangle inequality implies 
  \[
	 \left\|h \cdot ghg^{-1} \cdot g^2hg^{-2} \cdot \ldots \cdot g^{k-1}hg^{1-k} \right\| \leq \delta.
  \]
  Therefore using our choice of $\delta$,
  \[
	  |r(hg) - r(g)| < \frac{1}{k} \cdot (\varepsilon + D) < \frac{2 \|h\|}{\delta} (\varepsilon + D) = \frac{2}{\delta} (\varepsilon + D) \cdot \dist(hg, g).
  \]
  
  \medskip
  
  Combining the results above, $r$ is globally Lipschitz.
\end{proof}

Therefore a non-Lipschitz quasimorphism is also non-continuous. (With some extra effort one may show it is nowhere continuous.) 
In what follows we will concentrate only on the Lipschitz property.

\section{Calabi quasimorphisms}\label{S:Cal}

Let $F_t : M \to \RR$, $t \in [0, 1]$ be a time-dependent smooth function with compact support. Define
\[\widetilde{\Cal} (F_t) = \int_0^1 \left( \int_M F_t \omega \right) \dd t.\] When $\omega$ is exact on $M$,
$\widetilde{\Cal}$ descends to a homomorphism $\Cal: \Ham(M) \to \RR$ which is called 
the Calabi homomorphism. For any $g \in \Ham(M)$, $\Cal(g) \leq \| g \|$ by its definition. 

\medskip

Equip the unit sphere $S^2$ with a symplectic form normalized by $\int_{S^2} \omega = 1$. ~\cite{En-Po:calqm} presents construction of a homogeneous quasimorphism $\Cal_{S^2} : \Ham (S^2) \to \RR$ based on a Floer-theoretical spectral invariant of $\Ham(S^2)$.
Being a spectral invariant, it is Lipschitz with respect to the Hofer norm.

A symplectic embedding $j: M \to S^2$ of a genus zero surface into a sphere induces the pullback quasimorphism $j^*(Cal_{S^2})$ on $\Ham(M)$. In ~\cite{Bi-En-Po:Calabi-qm} the authors
show that this construction provides continuum of linearly independent quasimorphisms, all Lipschitz in the Hofer norm.
These quasimorphisms provide useful tools to analyze the Hofer geometry of $\Ham (M)$. For example, existence of such quasimorphisms allows to construct quasi-isometric embeddings of
``large'' sets into $\Ham(M)$ (see ~\cite{En-Po:calqm}) or into spaces of Hamiltonian-isotopic Lagrangian submanifolds of $M$ (e.g. ~\cite{Kh:diam, Sey:unb_lagr_ball}).

These quasimorphisms can be seen as generalizations of the Calabi homomorphism as they coincide with the Calabi invariant given a Hamiltonian isotopy which is supported in a 
displaceable subset. There was a hope that other quasimorphisms (some of them can also be seen as certain generalizations of Calabi) may share geometric properties with the Calabi homomorphism, 
for example, turn to be Lipschitz. Unfortunately, it is not the case. We analyze two constructions of quasimorphisms in the sections below.

\section{Polterovich construction}\label{S:Polt}
\subsection{The quasimorphisms}
Let $(M, \omega)$ be a symplectic manifold with non-abelian fundamental group. We equip it with an auxiliary Riemannian metric and pick a basepoint $z \in M$. 
Given a non-trivial homogeneous quasimorphism $r: \pi_1 (M, z) \to \RR$ one constructs a quasimorphism $\rho : \Ham(M) \to \RR$ as follows.
For each $x \in M$ choose a short geodesic path $\gamma_x$ which connects $x$ to $z$. Given a Hamiltonian isotopy $h_t$ between $\Id$ and $h \in \Ham(M)$, for each $x \in M$ denote by 
$l_x (h_t) = \gamma_x * \{h_t(x)\} * -\gamma_{h(x)}$ the closed loop defined by concatenation of geodesic paths with the trajectory of $x$ under $h_t$.
Let
\[
  \hat{\rho}(h_t) = \int_M r([l_x (h_t)]) \omega
\]
where $[l_x (h_t)] \in \pi_1 (M, z)$ is the equivalence class of $l_x (h_t)$. $\hat{\rho}: \widetilde{\Ham (M)} \to \RR$ is a quasimorphism and its homogenization
\[
  \rho(h_t) = \lim_{k \to \infty} \frac{1}{k} \int_M r([l_x ((h_t)^k)]) \omega
\]
does not depend on the choices of metric, $z$, geodesic paths or $h_t$, resulting in a homogeneous quasimorphism $\rho: \Ham(M) \to \RR$. For more details see ~\cite{Po:HF-dyn-gps}.

\subsection{Noncontinuity}
We start with a simple example which will be extended to the general case later.
Let $M$ be a closed surface of genus two equipped with a symplectic form $\omega$. 
Denote by $a, b, a', b'$ the standard generators of $\pi_1 (M, *)$ and pick a quasimorphism $r:\pi_1 (M, *) \to \RR$ such that $r (a) = r(b) = 0, r(aba^{-1}b^{-1}) = 1$
(for example, a word counting quasimorphism). Let $\rho : \Ham(M) \to \RR$ be the quasimorphism induced by $r$.
Denote by $P$ a pair of pants whose boundary components represent the free homotopy classes of $aba^{-1}b^{-1}, a, a^{-1}$.

\begin{figure}[!htbp]
\begin{center}
\includegraphics[width=0.4\textwidth]{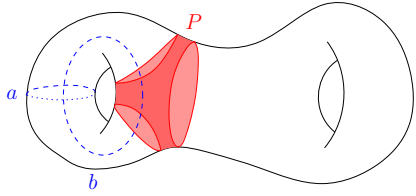}
\caption{Pair of pants}
\end{center}
\end{figure}

We construct a Hamiltonian $H$ from the indicator function of $P$ which is cut off near the boundary. Note that the flow $\phi_H^t$ generated by $H$ is controlled by the cutoff and is 
supported in a tubular neighborhood of the boundary $\partial P$. For a reasonable choice of the cutoff (when the level sets $\{H^{-1} (c)\}_{c \in (0,1)}$ have exactly one connected component near 
each connected component of $\partial P$) support of the flow consists of three strips around $\partial P$. These strips are foliated by periodic trajectories in the 
free homotopy classes of $a, a^{-1}, aba^{-1}b^{-1}$. The classes $a, a^{-1}$ do not affect the value of $\rho (\phi_H^t)$ and contribution comes only from trajectories in the class $aba^{-1}b^{-1}$.
A simple computation shows that $\rho (\phi_H^t) = t$ and it is independent of the cutoff.
(Intuitively, when one applies a steeper cutoff, velocity of the flow increases linearly with the slope and inversely proportional to its area of support, 
so faster rotation is compensated by the reduced volume of motion.)

We adjust the cutoff so that nearly all non-stationary points near $\partial P$ have a fixed rational period. Pick an area preserving parametrization $S^1 \times (-\varepsilon, \varepsilon)$ with coordinates 
$(\theta, h)$ near each connected component of $\partial P$ (in our notation $S^1 = \RR/\ZZ$). Suppose that $P$ is defined in these coordinates by $\{h \geq 0\}$. 
Pick a natural number $n > 1/\varepsilon$ and a smoothing parameter $\delta << 1/n$, define the cutoff function $c : (-\varepsilon, \varepsilon) \to [0, 1]$ by
\[
  c (h) = \left\{ \begin{array}{ll} 0 & \textrm{if $h \leq 0$} \\ hn  & \textrm{if $\delta \leq h \leq \frac{1}{n} - \delta$} \\
													   1 & \textrm{if $h \geq \frac{1}{n}$}
                          \end{array} \right.
\]
and interpolate it smoothly in the intervals $\left\{0 \leq h \leq \delta\right\}, \left\{\frac{1}{n} - \delta \leq h \leq \frac{1}{n}\right\}$.
Let $H (\theta, h) = c (h)$.
As the result, all points in the annulus $\left\{ \delta \leq h \leq \frac{1}{n} - \delta\right\}$ have rational period $T = 1/n$ and the remaining non-stationary points in $P$ 
are contained in a region of area $6 \delta$ and have controlled dynamics in a sense that their period is at least $1/n$. Note that the smoothing parameter $\delta$ can be chosen 
arbitrarily small and can be reduced independently of $P$ and $n$. Such adjustment of $\delta$ does not increase the maximal velocity of the flow.

\medskip

\underline{Step I:}
Given $N >> 1$ consider a pair of pants $P$ as above with $\Area(P) << \frac{1}{N^2}$. 
We place $N$ deformed copies $P_1 = h_1(P), \ldots, P_N = h_N(P)$ of $P$ (where the deformations $h_i$ are symplectic near $P$ and isotopic to the identity) 
in a way so that no triple intersections of $P_i$ occur.

\begin{figure}[!htbp]
\begin{center}
\includegraphics[width=0.5\textwidth]{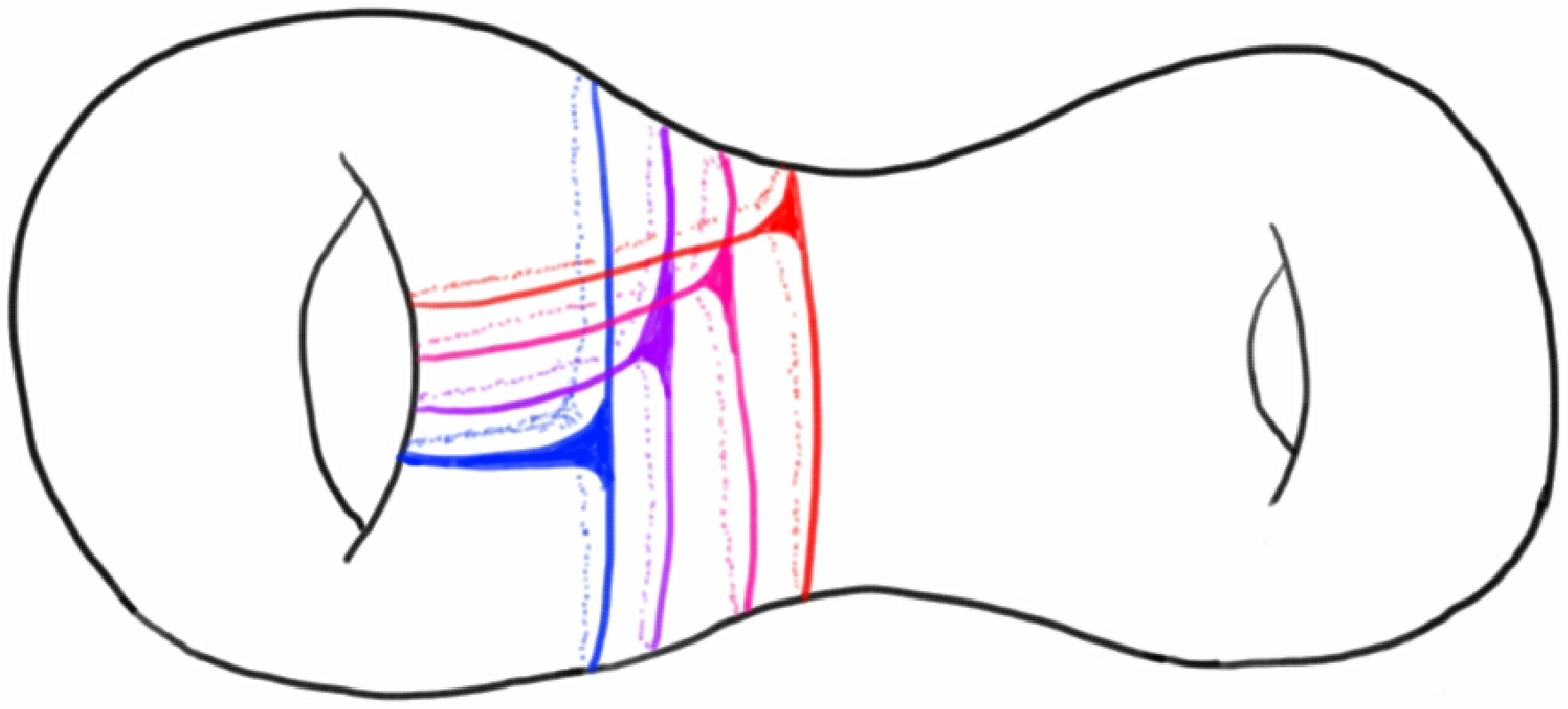}
\caption{}
\end{center}
\end{figure}

Pick a Hamiltonian function $H$ supported in $P$ as described above. The cutoff of $H$ is adjusted in a way that most (up to a region of area $6\delta << \Area (P)$) 
non-stationary points have the same period $T = 1/n$ for some $n \in \NN$. Denote the resulting Hamiltonian by $H_P$ and let $H_i$ be the deformations of $H_P$ 
supported in each $P_i$ defined by $h_i^* (H_i) = H_P$. Let $\phi_i^t = h_{i,*} \phi^t_{H_P}$ denote their Hamiltonian flows. 

\underline{Step II:}
Pick a system of disjoint open neighborhoods $\left\{ U_{ij} \right\}_{1 \leq i < j \leq N}$ such that $P_i \cap P_j \subset U_{ij}$. For $i > j$ let $U_{ij} = U_{ji}$.
We assume that in each connected component $Q$ of $U_{ij}$ there is enough ``empty space'' not occupied by $P_i$ and $P_j$: $\Area (Q) > 3 \Area ((P_i \cup P_j) \cap Q)$. 
Otherwise we may either repeat the construction applying the same deformations $h_i$ to a narrower pair of pants $P$ or deform the symplectic form $\omega$ in the complement of the union of all pairs of pants and 
redistribute area to satisfy this condition.

Let $t_0$ be the minimal time required for any of the flows $\phi_i$ to travel between two neighborhoods $U_{kl}$ and $U_{k' l'}$.
Namely, given $t < t_0$, $\phi_i^t (U_{k l}) \cap U_{k' l'} = \emptyset$ for all $\{k,l\} \neq \{k',l'\}$ and all $1 \leq i \leq N$. 
Pick a natural $m$ such that $\tau := \frac{T}{m} = \frac{1}{mn} < \frac{t_0}{2}$.

\underline{Step III:}
We make a series of adjustments to this construction. The goal is to ensure that most points are $T$-periodic while the rest (that are beyond our direct control) will not travel too far, 
hence their contribution to the quasimorphism will be small.

First of all, we modify the smoothing parameter $\delta$ that appears in the construction of $H_P$.
Fix an auxiliary Riemannian metric on $M$. Let $v$ be the maximal point velocity of the flows $\phi_i$ and $d_Q$ be the maximal diameter of 
connected components $Q$ of the neighborhoods $U_{ij}$. 
Let $$S_\tau = \{[\gamma] \in \pi_1 (M, *) \, \big| \, \len(\gamma) \leq \tau v + d_Q + 2 \diam (M)\}$$ be a set of homotopy classes of short loops and 
$r_\tau = \max \{r (c) \, \big| \, c \in S_\tau\}$ be the maximal value of the quasimorphism $r:\pi_1 (M, *) \to \RR$ 
restricted to $S_\tau$. (Short loops represent a finite number of homotopy classes so the maximum is attained.)
We pick 
\[
	\delta < \min \left( \frac{\tau}{20 (r_\tau + D) N}, \; \frac{\tau m}{3 N} \right)
\]
where $D$ is the defect of $r$. $H_i$ and the deformed flows $\phi_i$ are adjusted according to this value of the smoothing parameter $\delta$.
Note that this modification does not increase the maximal velocity of the flow. Hence $\tau$ is still less than the travel time between different neighborhoods $U_{kl}$.
 
\begin{figure}[!htbp]
\begin{center}
\includegraphics[width=0.65\textwidth]{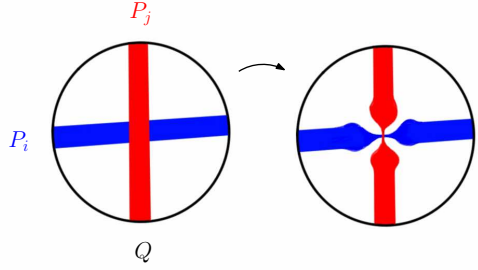}
\caption{Local adjustments}
\end{center}
\end{figure}

Furthermore, perturb $P_1, \ldots, P_N$ inside the neighborhoods $U_{ij}$ so that $$\Area (P_i \cap P_j) \leq \frac{\delta}{mN}.$$
We apply the same perturbations to the Hamiltonian functions $H_i$ and the flows $\phi_i$. This modification does not affect dynamics of the flows outside the neighborhoods $U_{ij}$, hence
the minimal travel time between neighborhoods is still greater than $\tau$. Moreover, the modified flows may have increased velocities locally inside $U_{ij}$, but on a larger scale (time $\tau$ and above) 
they travel approximately the same distance.

\medskip

Let $\Phi^t = \phi_1^t \circ \ldots \circ \phi_N^t$ be the composition of the time-$t$ maps of the flows.
We claim that for $t = \tau$, $$\rho (\Phi^\tau) = N \tau + O (\delta) > (N - 2) \tau $$ while the Hofer norm is bounded by $2 \tau$. Taking $N \to \infty$, we deduce that $\rho$ is not Lipschitz.

\begin{figure}[!htbp]
\begin{center}
\includegraphics[width=0.6\textwidth]{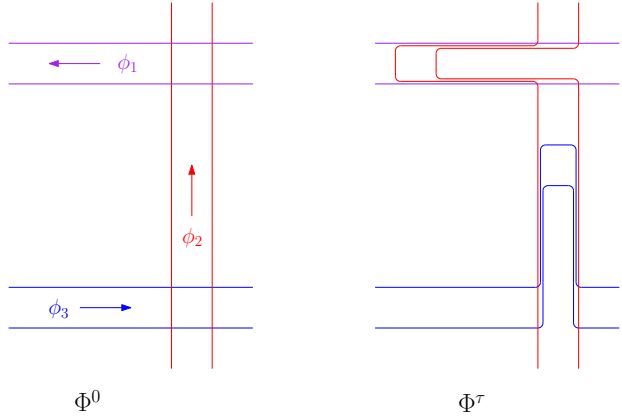}
\caption{The composition}
\end{center}
\end{figure}

By the composition formula, the flow $\Phi^t$ is generated by the time-dependent Hamiltonian 
\[
  G (t) = H_1 + \left( \phi_1^{-t} \right)^* (H_2) + \ldots + \left( \phi_1^{-t} \circ \ldots \circ \phi_{N-1}^{-t} \right)^* (H_N)
\]
Namely, the pairs of pants $P_i$ which support the Hamiltonian functions $H_i$ are deformed and ``dragged'' along the flows. For time $t < t_0$ supports of the summands cannot have triple intersections:
that happens only after some $P_i$ is dragged by $\phi_j^t$ to a point where it intersects some (possibly deformed) $P_k$. That means that certain points $p \in U_{ij}$ have arrived in 
(or passed through) other neighborhood $U_{jk}$ and that cannot happen before $t = t_0$. As each summand admits values in the interval $[0, 1]$, $\max G_t - \min G_t \leq 2$ for all $0 \leq t < t_0$ and the bound 
on the Hofer norm follows.

\medskip

We estimate the value of the quasimorphism $\rho (\Phi^\tau)$. 
Intuitively speaking, the time-$\tau$ maps $\{ \phi_i^\tau\}$ ``almost commute'' (commute in the complement of a subset of area $7 N \delta$), hence 
their contribution to the quasimorphism ``almost adds up'' (up to a defect which is small compared to the values $\rho (\phi_i^\tau)$ and is controlled by $\delta$).

Pick $k \leq N$. We consider three $\phi_k^\tau$-invariant regions in $M$ according to the dynamics of points under $\phi_k^\tau$:
\begin{enumerate}
\item 	
  points outside $P_k$ and those points inside that are stationary under $\phi_k$. They do not contribute to $\rho (\phi_k^\tau)$.
\item 	
  the set $A_k$ of $m$-periodic points under $\phi_k^\tau$ that never hit any of the remaining pairs of pants $\{P_i\}_{i \neq k}$.
  
  All non-stationary points are contained in the three cutoff strips near $\partial P_k$. Each strip has area $\frac{1}{n}$ and we have to remove those points
  that are not $m$-periodic due to smoothing. Thus the set of $m$-periodic points has area bounded between $\frac{3}{n} - 6 \delta$ and $\frac{3}{n}$.
  After removing points whose trajectories intersect $\{P_i\}_{i \neq k}$ we get the estimate
  \[
	\Area (A_k) \geq \frac{3}{n} - 6 \delta - m  \sum_{i \neq k} \Area (P_i \cap P_k) \geq \frac{3}{n} - 6 \delta - mN \cdot \frac{\delta}{mN} = \frac{3}{n} - 7 \delta
  \]
  Points of $A_k$ are distributed in three strips near $\partial P_k$ and their trajectories, depending on the strip, represent either $a, a^{-1}$ or $aba^{-1}b^{-1}$. $a$ and $a^{-1}$
  do not affect $\rho (\phi_k^\tau)$ while the set $A'_k$ of points in the class $aba^{-1}b^{-1}$
  has $\Area (A'_k) \geq \frac{1}{n} - 3 \delta$ and contributes
  \[
	  \frac{\Area(A'_k) \cdot r(aba^{-1}b^{-1})}{m} \geq \frac{1}{nm} - \frac{3 \delta}{m} = \tau - \frac{3 \delta}{m}
  \]
  
\item
  the set $B_k$ containing the rest of points (non-stationary points whose period is different from $m$ or whose trajectory under $\phi_k^\tau$ hits other pairs of pants).
  By a similar computation, $\Area (B_k) \leq 7 \delta$. While we have less control on the dynamics of $B_k$, we note that the velocity of points in $B_k$ under the flow $\phi_k$ 
  is bounded by $v$ outside neighborhoods $Q$ of the intersections $P_k \cap P_j$ (it can still be very large inside $Q$ because of the adjustments). Also note that the time needed to travel 
  between different neighborhoods is larger than $\tau$. Hence for each $p \in B_k$, the trajectory $\{ \phi_k^t (p)) \}_{t \in [0, \tau]}$ has length at most $\tau v + d_Q$.
  
  Applying stabilization, the trajectory of $p$ under $(\phi_k^{\tau})^K$ has length bounded by $K(\tau v + d_Q)$ and $\left[l_p \left( (\phi_k^\tau)^K\right)\right]$ can be presented as a product
  of at most $K$ elements from $S_\tau$, so $r ([l_p ( (\phi_k^\tau)^K)]) \leq K (r_\tau + D)$.
  The contribution of $B_k$ to $\rho (\phi_k^\tau)$ is bounded in the absolute value by
  \[
	  \lim_{K \to \infty} \frac{1}{K} \int_{B_k} |r([l_p ((\phi_k^\tau)^K)])| \omega \leq (r_\tau + D) \cdot \Area (B_k) = 7 \delta  (r_\tau + D).
  \]
\end{enumerate}

Now note that for all $k \leq N$, $A_k$ is $\Phi^\tau$-invariant and the action of $\Phi^\tau$ coincides with that of $\phi_k^\tau$. Indeed, for a point $p \in A_k$, 
$(\Phi^{\tau})^K (p) = (\phi_k^{\tau})^K (p)$ for all $K$ as $p$ never hits the pairs of pants $\{P_i\}_{i \neq k}$. More than that, the trajectory $\{\Phi^t (p)\}_{t \in [0, \tau]}$ is homotopic
relative endpoints to $\{\phi_k^t (p)\}_{t \in [0, \tau]}$ (other ingredients of $\Phi^t (p) = \phi_1^t \circ \ldots \circ \phi_N^t (p)$ can be homotoped away). 
The same is true for iterations of $\Phi^\tau$, hence $A_k$ contributes to $\rho (\Phi^\tau)$ the same amount as to $\rho (\phi_k^\tau)$.
We finish with an observation that the sets $\{A_k\}_{k = 1}^N$ are disjoint and do not intersect $B = \bigcup_{j=1}^N B_j$.

Let $p \in B$. If $p \in P_N$, the trajectory $\{\Phi^t (p)\}_{t \in [0, \tau]}$ is the trajectory $\{\phi_N^t (p)\}_{t \in [0, \tau]}$ which was possibly deformed by the flows
$\phi_{N-1}, \ldots, \phi_1$. Such deformation occurs each time $\{\phi_N^t (p)\}_{t \in [0, \tau]}$ crosses a pair of pants $P_{N-1}, \ldots, P_1$. Within time $\tau$
only one such crossing is possible. As the result, $\{\Phi^t (p)\}_{t \in [0, \tau]}$ is homotopic (relative endpoints) to a path of length at most $(2 \tau v + d_Q)$ 
(the path can be longer than $\tau v + d_Q$ when $\{\phi_N^t (p)\}_{t \in [0, \tau]}$ either starts or ends at the intersection with other pair of pants).
The same argument can be applied to points $p \in B \cap (P_{N-1} \setminus P_N)$, $p \in B \cap (P_{N-2} \setminus (P_N \cup P_{N-1}))$ and so on by induction. 
After taking $K$ iterations, trajectory of any $p \in B$ is homotopic relative endpoints to a path shorter than $2K (\tau v + d_Q)$, thus the contribution of $B$
to $\rho (\Phi^\tau)$ in the absolute value is at most
\begin{eqnarray*}
 \lim_{K \to \infty} \frac{1}{K} \int_{B} |r([l_p ((\Phi^\tau)^K)])| \omega &\leq& 2 (r_\tau + D) \cdot \Area (B) \leq 2 (r_\tau + D) \cdot \sum_{i=1}^N \Area (B_i) \\
			  &\leq& 14 N \delta  (r_\tau + D).
\end{eqnarray*}

As a summary, points from $A = \bigcup_{i=1}^N A_i$ contribute to  $\rho (\Phi^\tau)$ at least $N \tau - \frac{3 N \delta}{m}$.
$B = \bigcup_{i=1}^N B_i$ alters the result by at most $14 N \delta  (r_\tau + D)$.
Points $p \in M \setminus (A \cup B)$ are stationary and do not contribute.
We deduce
\[
  \rho (\Phi^\tau) \geq N \tau - \frac{3 N \delta}{m} - 14 N \delta  (r_\tau + D) > (N-2) \tau.
\]

\subsection{General case}

Let $M$ be a symplectic surface of finite type and suppose 
$r: \pi_1 (M, *) \to \RR$ is a genuine homogeneous quasimorphism. Pick $a, b \in \pi_1 (M, *)$ such that 
$r (ab) \neq r (a) + r (b)$ (without loss of generality assume that $d_r = r (a) + r (b) - r (ab) > 0$). 
Let $\alpha, \beta$ be based loops representing $a, b$, respectively. Denote by $\rho$ be the induced quasimorphism on $\Ham (M)$.
Let $\gamma = \alpha * \beta$. We wish to construct a Hamiltonian diffeomorphism $\Phi^\tau$ such that $\rho (\Phi^\tau) >> \|\Phi^\tau\|_H$. 

\underline{Step I:}
Pick a small tubular neighborhood $P$ of $\gamma$.
Applying $C^0$-small perturbations to $\alpha, \beta, \gamma$ we may assume that the three curves intersect (and self-intersect) 
transversely in a finite number of points and still lie in $P$.

\underline{Step II:}
Consider $N$ deformed copies ($N >> 1$) $P_1 = h_1 (P), \ldots, P_N = h_N (P)$ of $P$ where the deformations $h_i$ are isotopic to the identity and area-preserving near $P$ so that:
\begin{itemize}
  \item 
	no triple intersections of $\{P_i\}_{i=1}^N$ occur
  \item
	the collection of loops $Q = \{\alpha_i = h_i (\alpha) \}_i \cup  \{\beta_i = h_i (\beta) \}_i \cup \{\gamma_i = h_i (\gamma) \}_i$ does not have triple intersection/self-intersection
	points, the total number of intersections is finite and they are all transverse.
\end{itemize}

\underline{Step III:}
Construct a time-dependent Hamiltonian flow $\Phi^t$ supported in a narrow tubular neighborhood of curves from $Q$ 
which translates points along $\{\alpha_i\}_{i=1}^N, \{\beta_i\}_{i=1}^N$ preserving orientation of curves 
and along $\{\gamma_i\}_{i=1}^N$ with reversed orientation.

\begin{figure}[!htbp]
\begin{center}
\includegraphics[width=0.6\textwidth]{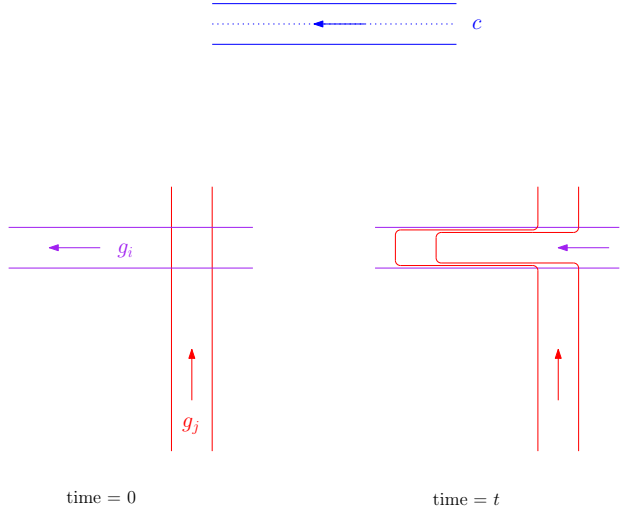}
\caption{Local picture}
\end{center}
\end{figure}

Local picture:
\begin{itemize}
  \item 
   near a curve $c \in Q$ and away from intersection points the flow is a parallel translation of a narrow strip around $c$ with fixed velocity and flux $1$. The flow is cut off in a 
   smooth way so that it remains parallel to $c$ and the area affected by the cutoff is small.
  \item
   near an intersection of two curves $c_i, c_j$ (or a self-intersection of two arcs $c_i, c_j$ of the same curve from $Q$): construct the flows $g_i, g_j$ parallel to $c_i, c_j$ as in the 
   previous case. We may symplectically perturb $g_i, g_j$ near $c_i \cap c_j$ to ensure that $\Area (\supp (g_i) \cap \supp (g_j))$ is sufficiently small. 
   Define the local flow $\Phi^t$ to be the composition $g_i^t \circ g_j^t$. Intuitively speaking, the curve $c_j$ and the flow $g_j$ following $c_j$ ``drift'' 
   with $g_i^t$. 
\end{itemize}
For $t$ small enough the flow charts can be adjusted and patched together to form a global symplectic time-dependent flow $\Phi^t$. More than that, using local adjustments, one can ensure that:
\begin{itemize}
  \item
	flow strips around $\alpha_i, \beta_i, \gamma_i$ are narrow enough and stay inside $P_i$.
  \item 
	local velocities are chosen in a way so that if one ignores all intersections, time $T$ needed to traverse a loop $c \in Q$ is the same for all loops in $Q$.
  \item
	the total area of the cutoff zones and intersections / self-intersections of the flow strips is at most $\delta$ ($\delta$ is a small parameter which will be specified later).
\end{itemize}
The construction described above works for $t \leq t_0$ where $t_0$ is the minimal time needed for an intersection point $p \in c_i \cap c_j$ to leave the appropriate chart 
when it follows the flows $g_i$ or $g_j$.

\medskip

Clearly, the flow $\Phi^t$ is area-preserving. We claim that it has zero flux, hence is Hamiltonian. Indeed, the flux along each $\alpha_i, \beta_i$ is $1$ while the flux along $\gamma_i$ is $-1$ due
to reversed orientation. A generic cycle $c$ in $(M, \partial M)$ satisfies $\#c \cap \alpha_i + \#c \cap \beta_i = \# c \cap \gamma_i$ (the intersection points are counted with signs), 
so the fluxes cancel out.
Denote by $H_t$ the Hamiltonian function which generates $\Phi^t$. It is a well known fact that $H_t (p) - H_t (q)$ equals the flux of $\Phi^t$ through a curve which connects $p$ to $q$
(for a Hamiltonian flow the flux depends on the endpoints and not on the curve itself). The argument above also shows that the flux between two points $p, q \notin \bigcup P_i$ is zero,
hence, up to a shift by appropriate constant, $\supp H_t \subseteq \bigcup P_i$.

For a generic curve $c$ connecting two points $p, q \in P_1 \setminus (\alpha_1 \cup \beta_1 \cup \gamma_1)$ let $K_c = \# c \cap \alpha_1 + \# c \cap \beta_1 - \# c \cap \gamma_1$ (count with signs).
Roughly, if one ignores the width of flow strips around the curves, $K_c$ measures the flux between $p$ and $q$ of the flows along $\alpha_1, \beta_1, \gamma_1$. 
$K = \max_c K_c$ equals the maximal flux between any two points in $P_1$. Fluxes of flows inside each $P_i$ are additive, so as long as no triple intersections are allowed and 
$t \leq t_0$ (meaning the curves do not drift far and combinatorics of their intersections does not change), the maximal flux of $\Phi^t$ (and hence the variation of $H_t$) is at most $2K$. 
That is, the Hofer length of the isotopy $\Phi^t$ is bounded by $2Kt$. Note that the bound $K$ may depend on the perturbations performed in Step I 
but does not depend on the choice of $N$ in Step II.
  
\medskip
  
Pick $\tau = \frac{T}{m}\leq t_0$ for some $m \in \NN$.
The flow strip around a curve $c \in Q$ has flux $1$ and period $T$, hence has area equal to $T$ (up to effects of smoothing which slow down the flow and are controlled by $\delta$).
We subtract points which are slowed by smoothing and also those whose trajectory under $\{(\Phi^\tau)^k\}_{k \in \ZZ}$ lands on the intersections/self-intersections of strips.
The remaining region of area $T - O(\delta)$ is $m$-periodic and contributes to $\rho (\Phi^\tau)$ either $$T \cdot \frac{r(a)}{m} + O(\delta) = \tau r(a) + O(\delta)$$
in the case $c = \alpha_i$, $\tau r(b) + O(\delta)$ if $c = \beta_i$ or $-\tau r(ab) + O(\delta)$ when $c = \gamma_i$ for some $i \leq N$.

The rest of non-stationary points (either affected by smoothing or those visiting intersections of strips) are contained in a region $B$ of area $O (\delta)$. 
The long-term dynamics of $p \in B$ can be complicated, but locally (till time $\tau$) $p$ can visit at most one intersection of strips, so the length of trajectory $\{\Phi^t (p)\}_{t \in [0, \tau]}$
is uniformly bounded. As the result, $\left|\frac{1}{k}r([l_p ((\Phi^\tau)^k)])\right|$ is also uniformly bounded by some $C > 0$ and the contribution of $B$ to $\rho (\Phi^\tau)$ is bounded 
in the absolute value by $$\Area (B) \cdot C = O(\delta).$$

Summing up, 
\[
  \rho (\Phi^\tau) = \sum_{i=1}^N (\tau r(a) + \tau r(b) - \tau r(ab) + O(\delta)) + O(\delta) = \tau N d_r + O(\delta).
\]
Now we pick $\delta = \delta(N)$ small enough to ensure $\rho (\Phi^\tau) > \tau (N-1) d_r$ 
and perform rearrangements of smoothing and intersection zones to comply with this chosen value of $\delta$. These local adjustments do
not affect $t_0$ or estimates of trajectory lengths, hence all computations we performed above remain in place. The claim follows by picking $N$ such that $N d_r >> 2 K$.

\subsection{Higher dimensions}
The construction above can be lifted to a symplectic ball bundle $\widehat{M} \to M$. We lift the flow $\Phi^t$ from $M$
and cut it off near $\partial \widehat{M}$. When the cutoff is steep enough, it affects only a small volume of motion while keeping dynamics under control
(cutoff keeps the $M$ component of the flow velocity bounded and introduces a fast rotation around the boundary in the fiber coordinates. This rotation does not affect average 
homotopy classes of orbits.)
Hence its effect on the quasimorphism $\widehat{\rho} : \Ham (\widehat{M}) \to \RR$ can be made arbitrarily small. 
The bounds on the Hofer length in $\widehat{M}$ remain the same as in $M$ and non-continuity follows.

Let $M$ be a higher dimension symplectic manifold, $r : \pi(M, *) \to \RR$ a genuine homogeneous quasimorphism.
Let $\alpha, \beta$ be based loops in $M$ such that $r ([\alpha * \beta]) \neq r([\alpha]) + r ([\beta])$.
Perturb $\alpha, \beta$ so that they intersect at one point and that near the intersection they lie in a symplectic $2$-disk $D$.
$D$ can be extended to a smooth symplectic surface $\Sigma$ with boundary which contains both $\alpha$ and $\beta$. By the symplectic neighborhood theorem, a small neighborhood
of $\Sigma$ in $M$ is symplectomorphic to a ball bundle over $\Sigma$. Now we may apply the construction for ball bundles and show non-continuity of the quasimorphism induced by $r$.

\begin{rmk}
  Lev Buhovsky has pointed out that in the case of ball bundles there is a much simpler construction of Hamiltonian flows which also shows non-continuity of quasimorphisms (~\cite{Bu:private-com}).
  However this construction does not apply in the two-dimensional case so we could not avoid the more complicated setup.
\end{rmk}

\section{Gambaudo-Ghys construction}\label{S:GG}

\subsection{The quasimorphisms}

Let $(M, \omega)$ be a symplectic surface of finite type and area $1$, equipped with an auxiliary Riemannian metric. Denote by $B_k$ the braid group with $k$ strands in $M$ and by $PB_k$ the pure braid group.
Suppose $\{g_t\} \in \Ham (M)$ is an isotopy starting at the identity and ending at $g$. 
Pick $k$ distinct points $w_1, \ldots, w_k$ in $M$ to be a basepoint $\mathbf{w}$ in the configuration space $X_k (M)$. For an $\mathbf{x} = (x_1, \ldots, x_k) \in X_k (M)$ we construct
$k$ based loops $l_{w_j,x_j} (g_t)$ by concatenating a short geodesic path $\gamma_{w_j, x_j}$ with the trajectory $\{g_t (x_j)\}$ and closing the loop by a short geodesic path 
$\gamma_{g(x_j), w_j}$.
For almost all $\mathbf{x} \in X_k$, the $k$-tuple of loops $l_\mathbf{x} (g_t) = \left(l_{w_1,x_1} (g_t), \ldots, l_{w_k,x_k} (g_t) \right)$ defines a based loop in $X_k (M)$.
$\pi_1 (X_k, \mathbf{w})$ can be identified with the pure braid group $PB_k$, hence $[l_\mathbf{x} (g_t)]$ defines an element in $PB_k$.

Let $r: B_k \to \RR $ be a homogeneous quasimorphism. Then
\[
  \hat{\rho}(g_t) = \int_{X_k} r\left([l_{\mathbf{x}} (g_t)]\right) \omega^k
\]
defines a quasimorphism $\hat{\rho}: \widetilde{\Ham (M)} \to \RR$ and its homogenization
\[
  \rho(g_t) = \lim_{m \to \infty} \frac{1}{m} \int_{X_k} r\left([l_{\mathbf{x}} ((g_t)^m)]\right) \omega^k
\]
does not depend on the choices of metric, basepoint $\mathbf{w}$, geodesic paths or $g_t$, resulting in a 
homogeneous quasimorphism $\rho: \Ham(M) \to \RR$. For a more detailed explanation see ~\cite{Bra:biinv-quas}.

This construction generalizes the quasimorphisms by Polterovich: $PB_1 (M) \simeq \pi_1 (M, *)$, therefore for $k=1$ both constructions agree.

\begin{rmk}
In the case of a disk $M = D^2$, $PB_2 (M) \simeq \ZZ$. In this case any homogeneous quasimorphism $r : PB_2 (M) \to \RR$ is a homomorphism and the induced 
quasimorphism $\rho : \Ham (D^2) \to \RR$ is a multiple of the Calabi homomorphism (see ~\cite{Ga-Ghys:enlace, Fa:thesis}).

This way, the Gambaudo-Ghys construction can be seen as a generalization of the Calabi homomorphism.
\end{rmk}

\subsection{Noncontinuity}

Denote by $F_k$ the subgroup of $PB_k$ defined by fixing strands $2$ to $k$ and letting the first strand wind around. 
$F_k$ fits into a short exact sequence
\[
  0 \to F_k \to PB_k \to PB_{k-1} \to 0
\]
where the map $PB_k \to PB_{k-1}$ is the Chow homomorphism which ``forgets'' the first strand of $B_k$.

We consider a particular family of quasimorphisms: suppose $r : B_k \to \RR$ is a homogeneous quasimorphism 
such that the restriction $r \big|_{F_k}$ is a genuine quasimorphism. We show that the quasimorphism $\rho: \Ham (M) \to \RR$ induced by the Gambaudo-Ghys construction is not Hofer-Lipschitz.

\begin{rmk}
  E. Shelukhin pointed out that in the case of surfaces with boundary, $PB_k$ may be presented as a semidirect product of $F_k$ with $PB_{k-1}$
  ($PB_{k-1}$ is embedded into $PB_k$ by adding a stationary strand near the boundary). In this case one can show by induction that every genuine homogeneous quasimorphism on $B_k$
  restricts to a genuine quasimorphism on $F_k$.
  
  If $M$ has no boundary, consider a punctured surface $\dot{M} = M \setminus \{p\}$. The natural inclusion $\dot{M} \to M$ induces a surjective homomorphism $j: B_k (\dot {M}) \to B_k (M)$ of braid groups.
  In this case the pullback quasimorphism $j^*r : B_k (\dot {M}) \to \RR$ satisfies the condition by the remark above, therefore its restriction $j^*r \big|_{F_k (\dot{M})}$ is a genuine quasimorphism.
  But $$j^*r \big|_{F_k (\dot{M})} = j\big|_{F_k (\dot{M})}^* \left( r \big|_{F_k (M)} \right),$$ hence $r \big|_{F_k (M)}$ is genuine as well. 
  
  Therefore this technical condition on $r \big|_{F_k}$ is automatically satisfied.
\end{rmk}

When $M = S^2$ and $k \leq 3$, $B_k$ is finite and all homogeneous quasimorphisms are trivial.
Otherwise $F_k$ is naturally isomorphic to the fundamental group of $k-1$ times punctured surface $M ' = M \setminus \{p_2, \ldots, p_k\}$ (see ~\cite{Bi:braids}).
Let $r' : \pi_1(M', *) \to \RR$ be the quasimorphism induced by $r \big|_{F_k}$. Pick $a, b \in \pi_1 (M', *)$
such that $r' (a) + r' (b) \neq r' (ab)$.

\medskip

We use construction from the previous section and apply a variation of the argument by T. Ishida ~\cite{Ish:quasi-braid}. Given $N >> 1$ and sufficiently small $\delta = \delta (N)$, 
consider a narrow region $P_N$ which contains loops representing classes $a, b, ab$. We place $N$ copies of 
$P_N$ in $M'$ without triple intersections and construct $\Phi_N^{\tau_N} \in \Ham (M')$. 
Let $g_N = \left(\Phi_N^{\tau_N}\right)^{m_N}$ where $m_N$ is the common period for most non-stationary points under $\Phi_N^{\tau_N}$.

$g_N$ has several invariant subsets:
\begin{itemize}
  \item 
	subsets $A_a, A_b, A_{ab}$ that are fixed by $g_N$. Trajectories of points in $A_a$ represent $a$, while $A_b$ and $A_{ab}$ represent $b$ and $-(ab)$, respectively.
	Each of the three subsets has area $N T_N + O (\delta) = N m_N \tau_N + O (\delta)$,
  \item
	subset $B$ with $\Area (B) = O (\delta)$. Trajectories in $B$ are beyond control and have topological complexity of $O(m_N)$,
  \item
	all remaining points in $M'$ are stationary under the flow.
\end{itemize}
In addition, the Hofer norm $\|g_N\|_{M'} \leq 2 m_N \tau_N = 2 T_N$ and $\delta$ can be chosen arbitrarily small without affecting the estimates above.

Pick $\varepsilon << 1$ and $d_2, \ldots d_k \in [0, 1]$ such that $\varepsilon + d_2 + \ldots + d_k = 1$.
Rescale the symplectic form of $M'$ so that $\Area (M') = \varepsilon$ and glue in disks $D_2, \ldots, D_k$ of area $d_2, \ldots, d_k$ in the punctures $p_2, \ldots, p_k$, respectively.
The resulting surface $\widehat{M}$ has area $1$ hence is symplectomorphic to the unpunctured surface $M$. $g_N$ induce Hamiltonian diffeomorphisms on $\widehat{M}$
which will be denoted by $g_N$ as well, and due to rescaling their Hofer norms satisfy $$\|g_N\|_{\widehat{M}} \leq 2 T_N \varepsilon.$$
 
We estimate $\rho (g_N)$. Consider possible braids that arise from trajectories of $\mathbf{x} = (x_1, \ldots x_k) \in X_k (M)$ by $g_N$ and their contribution to $\rho (g_N)$:
\begin{itemize}
  \item 
	all $k$ points belong to the disks $\{D_j\}$, hence are stationary. Such braids are trivial hence do not affect the value of $\rho$.

  \item 
	one point lies in $M'$ and exactly one point belongs to each disk $D_j$, $2 \leq j \leq n$. 
	Order of points does not matter: reordering of points corresponds to permutation of strands.
	$r$ is homogeneous, hence invariant under conjugations in $B_k$ that permute stands.
	Without loss of generality, assume $x_1 \in M'$. $l_\mathbf{x} (g_N)$ is induced by the trajectory of $x_1$ in $M'$ and $k-1$ stationary strands in the disks. 
	$$r([l_\mathbf{x} (g_N)]) = r'([l_{x_1} (g_N)])$$ 
	(the same is true for iterations of $g_N$).
	$A_a \cup A_b \cup A_{ab}$ contribute 
	\[
	  \left(N T_N  + O(\delta)\right) \left(r'(a) + r' (b) - r'(ab)\right) \varepsilon d_2 \ldots d_k
  	\]
	while $B$ contributes $$O (\delta) \cdot O \left(m_N\right) \varepsilon d_2 \ldots d_k = O (\delta) \varepsilon d_2 \ldots d_k$$ as $\delta$ can be chosen independently of $m_N$.
	Stationary points in $M'$ induce a trivial braid and do not contribute anything.
	The total contribution is
	\[
	  \left(k! N T_N (r'(a) + r' (b) - r'(ab)) + O(\delta)\right) \varepsilon d_2 \ldots d_k = 
	\]
	\[
	  = \left(N T_N R_{1,2,\ldots, d_k}+ O(\delta)\right)\varepsilon d_2 \ldots d_k.
	\]
	$k!$ takes into account reorderings of points and the resulting coefficient $$R_{1,2,\ldots, d_k} \neq 0$$.
	
  \item 
	$x_i \in M'$ for some $1 \leq i \leq k$ while the rest belong to disks $\{D_j\}$ and the correspondence between points and disks is not bijective.
	Let $i_2, \ldots, i_k$ denote the indices of disks for the points $\{ x_j \}_{j \neq i}$ in the ascending order of $j$. 
	The contribution of this configuration of points is given by
	$$R_{N T_N, i, i_2, \ldots, i_k} \varepsilon d_{i_2} \ldots d_{i_k}$$
	for some coefficient $R_{N T_N, i, i_2, \ldots, i_k} \in \RR$. Similarly to the previous case, the dynamics of $g_N$ implies that this coefficient scales linearly with $N T_N$ (up to $O(\delta)$), 
	so 	we may rewrite
	$$R_{N T_N, i, i_2, \ldots, i_k} \cdot \varepsilon d_{i_2} \ldots d_{i_k} = (N T_N R_{i, i_2, \ldots, i_k} + O(\delta)) \varepsilon d_{i_2} \ldots d_{i_k}.$$
	Summing over all configurations, we get the following polynomial in $\varepsilon, d_2, \ldots, d_k$
	\[
	  \sum_{\begin{subarray}{c} 1 \leq i \leq k \\ 2 \leq i_2, \ldots, i_k \leq k \end{subarray}} (N T_N R_{i, i_2, \ldots, i_k} + O(\delta)) \varepsilon d_{i_2} \ldots d_{i_k}
	\]
	which does not contain the monomial $\varepsilon d_2 \ldots d_k$.

  \item 
	two points or more belong to $M'$, the rest lie in disks $\{D_j\}$. Topological complexity of the strand is $O\left(m_N\right)$ (this bound is independent of $\delta$) and the contribution
	is 	$O\left(m_N \cdot \varepsilon^2\right)$ (areas $d_i$ are bounded by $1$ and can be incorporated into $O\left(m_N \cdot \varepsilon^2\right)$).
\end{itemize}

So the total value of the quasimorphism $\rho (g_N)$ is
\[
  	N T_N \varepsilon \left[ \left(R_{1,2,\ldots, d_k}+ \frac{O(\delta)}{N T_N}\right) d_2 \ldots d_k + \sum \left(R_{i, i_2, \ldots, i_k} + 
  	\frac{O(\delta)}{N T_N}\right) d_{i_2} \ldots d_{i_k} \right] + O\left(m_N \cdot \varepsilon^2\right)
\]
The polynomial $R_{1,2,\ldots, d_k} d_2 \ldots d_k + \sum R_{i, i_2, \ldots, i_k} d_{i_2} \ldots d_{i_k}$ is not zero, so one may find
values for $d_2, \ldots, d_k$ in the simplex $\{d_i \geq 0, \; \sum d_i < 1\}$ where the polynomial does not vanish. 
Pick $N >> 1$. $\delta$ and $\varepsilon$ are independent parameters. As soon as we fix $N$ (hence also $T_N$ and $m_N$), we let $\delta, \varepsilon \to 0$ while $d_2, \ldots, d_k$ 
are rescaled keeping their ratios constant and the total area 
$$\Area \left(\widehat{M}\right) = \varepsilon + d_2 + \ldots + d_k = 1.$$ 
In the limit $\rho (g_N)$ is proportional to $N T_N \varepsilon$ while Hofer's norm is bounded by $2 T_N \varepsilon$. It follows that $\rho$ cannot be Lipschitz.


\bibliography{bibliography}

\end{document}